\documentclass[preprint,hidelinks,onefignum,onetabnum]{siamart220329}

\usepackage{lipsum}
\usepackage{amsfonts}
\usepackage{graphicx}
\usepackage{epstopdf}
\usepackage{algorithmic}
\ifpdf
  \DeclareGraphicsExtensions{.eps,.pdf,.png,.jpg}
\else
  \DeclareGraphicsExtensions{.eps}
\fi

\headers{On a spherically lifted spin model}{X. Tang, Y.Khoo , L. Ying}

\title{On a spherically lifted spin model at finite temperature}

\author{
  Xun Tang\thanks{Department of Mathematics, Stanford University, Stanford, CA 94305, USA. 
  (\email{xuntang@stanford.edu})}
  \and
  Yuehaw Khoo\thanks{Department of Statistics, University of Chicago, Chicago, IL 60637, USA. 
  (\email{ykhoo@uchicago.edu})}
  \and
  Lexing Ying\thanks{Department of Mathematics and Institute for Computational and Mathematical Engineering (ICME), Stanford University, Stanford, CA 94305, USA. 
  (\email{lexing@stanford.edu})}
  \funding{X.T. and L.Y. are supported by AFOSR MURI award FA9550-24-1-0254.}
  }

\usepackage{amsopn}

\makeatletter
\newcommand*{\addFileDependency}[1]{%
  \typeout{(#1)}%
  \@addtofilelist{#1}%
  \IfFileExists{#1}{}{\typeout{No file #1.}}%
}
\makeatother

\usepackage{amsmath,amsfonts,bm}

\def\eqref#1{equation~\ref{#1}}

\def\1{\bm{1}}

\def\eps{{\epsilon}}

\DeclareMathAlphabet{\mathsfit}{\encodingdefault}{\sfdefault}{m}{sl}
\SetMathAlphabet{\mathsfit}{bold}{\encodingdefault}{\sfdefault}{bx}{n}

\newcommand{\R}{\mathbb{R}}

\DeclareMathOperator*{\argmax}{arg\,max}

\usepackage{amsmath,amssymb,a4wide}
\usepackage{latexsym}
\usepackage{indentfirst}
\usepackage{placeins}
\usepackage{enumerate}
\usepackage{booktabs}
\usepackage{algorithm}
\usepackage{algorithmic}
\usepackage{multirow}
\usepackage{optidef}
\usepackage{graphicx,color}
\usepackage{diagbox}
\usepackage[normalem]{ulem}
\usepackage{hyperref}
\usepackage{cleveref}

\newtheorem{thm}{Theorem}
\newtheorem{prop}{Proposition}[section]

\newcommand{\logdet}{\mathrm{logdet}}

\newcommand{\splus}{\mathcal{S}^{+}}
\newcommand{\sGW}{\mathcal{G}}
\newcommand{\fstargw}{q^{\star}}
\newcommand{\sstargw}{S^{\star}}

\newcommand{\gmax}{g_{\mathrm{max}}}

\DeclareMathOperator*{\maximize}{maximize}

\begin{document}

\maketitle
\begin{abstract}
    We investigate an \(n\)-vector model over \(k\) sites with generic pairwise interactions and spherical constraints. The model is a lifting of the Ising model whereby the support of the spin is lifted to a hypersphere. We show that the \(n\)-vector model converges to a limiting distribution at a rate of \(n^{-1/2 + o(1)}\). We show that the limiting distribution for \(n \to \infty\) is determined by the solution of an equality-constrained maximization task over positive definite matrices. We prove that the obtained maximal value and maximizer, respectively, give rise to the free energy and correlation function of the limiting distribution. In the finite temperature regime, the maximization task is a log-determinant regularization of the semidefinite program (SDP) in the Goemans-Williamson algorithm. Moreover, the inverse temperature determines the regularization strength, with the zero temperature limit converging to the SDP in Goemans-Williamson.
    Our derivation draws a curious connection between the semidefinite relaxation of integer programming and the spherical lifting of sampling on a hypercube. To the authors' best knowledge, this work is the first to solve the setting of fixed \(k\) and infinite \(n\) under unstructured pairwise interactions.
\end{abstract}
\section{Introduction}\label{sec: intro}

This work focuses on a ubiquitous class of vector-spin models over \(k\) sites. Each site \(i \in \{1, \ldots, k\}\) is associated with a spin \(x_i\) supported on the \((n-1)\)-dimensional sphere \(\mathbb{S}^{n-1} \subset \R^{n}\). The model we consider is the Boltzmann distribution under an inverse temperature \(\beta > 0\) and a symmetric matrix \(A \in \R^{k \times k}\). The unnormalized distribution function of the model is
\begin{equation}\label{eqn: spherical model}
    p(x_{1}, \ldots, x_{k}) = \exp{\left(\beta n\sum_{i,j = 1}^{k}\left<x_{i}, x_{j}\right>A_{ij}\right)}\prod_{i=1}^{k}\delta(1 - \lVert x_{i}\rVert^2 ),
\end{equation}
where \(A\) encodes the pairwise interaction. One can see from \Cref{eqn: spherical model} that \(p\) is \(O(n)\)-invariant, and \(p\) is commonly referred to as the \(n\)-vector model \cite{stanley1968spherical, makeenko2002methods}. For example, the cases where \(n = 1\) and \(n = 2\) are commonly referred to as the Ising model and the Potts model. The \(n\)-vector model can be seen as a lifting of the Ising model corresponding to \(n = 1\).

Similar vector-spin models have been studied extensively. The work in \cite{stanley1968spherical} shows that the \(n \to \infty\) limit is solvable when \(A\) represents an isotropic lattice model. Subsequently, \cite{stanley1969exact} solves the isotropic 1D spin-chain model with arbitrary \(n\). A more general case is considered in \cite{pastur1982disordered} where \(A\) represents a disordered lattice model with finite range interaction and the model is studied under the \(n \to \infty\) limit.

This work focuses on the \(n > k\) case, which is an ``over-parameterized" regime where the spin dimensionality is greater than the number of sites. This work shows that the \(n \to \infty\) limit is exactly solvable for a general interaction \(A\). To the best knowledge of the authors, this work is the first theoretical treatment for general \(A\) at the \(n \to \infty\) limit.

A key element of this work is to show that the model in \Cref{eqn: spherical model} is intricately connected to the semidefinite relaxation for the associated max-cut problem belonging to integer programming. The particular optimization task of interest is formulated as follows:
\begin{equation}\label{eqn: Goemans-Williamson with logdet}
    \begin{aligned}
       &\maximize_{S \in \R^{k \times k}} 
       \quad\beta\mathrm{tr}(AS) +  \frac{1}{2}\logdet(S)\\
       &\text{subject to} \quad S \succeq 0, \,\,S_{ii} = 1, \,\, i = 1, \ldots,k.
    \end{aligned}
\end{equation}
One can see that \Cref{eqn: Goemans-Williamson with logdet} is the semidefinite program (SDP) in the Goemans-Williamson algorithm \cite{goemans1995improved} with a log-determinant regularization term, and the \(\beta \to \infty\) limit corresponds to the unregularized case. Numerically solving \Cref{eqn: Goemans-Williamson with logdet} is efficient by using the conventional primal-dual method \cite{vandenberghe1996semidefinite}. We show that the maximal value and maximizer to \Cref{eqn: Goemans-Williamson with logdet} exactly correspond to the free energy and correlation function of the \(n\)-vector model in \Cref{eqn: spherical model} under the \(n \to \infty\) limit. The result exhibits a clear connection between the \(n\)-vector model and convex optimization over positive definite matrices.

We give one motivation for studying the \(n\)-vector model through its link to the study of over-parameterization and convex relaxation in optimization. There is a widely acknowledged folklore that over-parameterization improves the optimization landscape for non-convex optimization problems. For example, the effect of over-parameterization has been extensively explored recently for training tasks in deep learning \cite{ge2016matrix, du2018gradient, du2018power, li2018learning, allen2019convergence, arora2019fine, zou2019improved, liu2022loss}, and similar effects are well-known in convex relaxation \cite{candes2010power,boumal2016nonconvex}. Similarly, the SDP in the Goemans-Williamson algorithm is derived from the associated max-cut problem by lifting the decision space from the binary spin values \((s_1, \ldots, s_{k}) \in  \{-1, 1\}^{k}\) to the spherical spin values \((x_1, \ldots, x_{k}) \in  \left(\mathbb{S}^{n-1}\right)^{k}\) for \(n > k\). The lifting relaxes the NP-complete max-cut problem to an efficient equality-constrained SDP task. 

From this perspective, this work is a novel study on the effect of over-parameterization in the setting of sampling from probability distributions. The main statements of this work show that over-parameterization substantially simplifies the sampling task when the probability density is the Boltzmann distribution from max-cut problems. In particular, \Cref{thm: characterization of spherical model} in \Cref{sec: main contribution} shows that the approximate sampling of \(p\) only requires one to solve the regularized Goemans-Williamson SDP in \Cref{eqn: Goemans-Williamson with logdet} and to generate random rotation matrices from the Haar measure of the rotation group \(O(n)\). Moreover, while free energy approximation is exponentially difficult for general Ising models, \Cref{thm: characterization of free energy} in \Cref{sec: main contribution} shows that one can approximate the free energy of the lifted \(n\)-vector model through solving \Cref{eqn: Goemans-Williamson with logdet}. 

\subsection{Main contributions}\label{sec: main contribution}
We summarize the main contributions of this work. We first go through the main notations. We write \((X_i)_{i=1}^{k} \sim p\) to denote that \((X_i)_{i=1}^{k}\) is distributed according to \(p\), where \(X_i\) is the \(n\)-vector spin at site \(i\). We write \(X = \left[X_1, \ldots, X_k\right]\) to denote the \(\R^{n \times k}\) random matrix obtained from a column-wise concatenation of the \(n\)-vector spin at each site. We also write \(X \sim p\) to mean that the columns of \(X\) are distributed according to \(p\). For a generic matrix \(M\), one can perform the QR factorization by the Gram-Schmidt algorithm to get \(M = QR\), where \(Q \in \R^{n \times k}\) has orthonormal columns and \(R \in \R^{k \times k}\) is upper-triangular. For invertible \(M\), we use \(\Phi(M) = (Q, R)\) to denote the unique output of the Gram-Schmidt algorithm so that the diagonal entries of \(R\) are positive. We use \(\lVert \cdot \rVert_{F}\) to denote the Frobenius norm.

\paragraph{Convergence of the \(n\)-vector model}
Our first main statement concerns the distribution of the \(n\)-vector model at \(n > k\). The following statement in \Cref{thm: characterization of spherical model} shows that the distribution of \(X\) can be characterized by a nice product measure under the Gram-Schmidt factorization \(\Phi(X) = (Q, R)\). We remark that \(X\) is generically invertible for \(X \sim p\), and so \(\Phi(X)\) is well-defined almost surely. In particular, the statement shows that the Gram-Schmidt factorization of \(X=QR\) has a uniformly distributed \(Q\) and an approximately deterministic \(R\). As a consequence of \Cref{thm: characterization of spherical model}, statistical moments of \(p\) can be approximated by considering the distribution of \(X = QR\) where \(Q\) is uniformly sampled and \(R\) is fixed at \(R = R^{\star}\).
\begin{thm}\label{thm: characterization of spherical model}
    Let \(X \sim p\) for \(p\) in \Cref{eqn: spherical model}, and let \((Q, R) = \Phi(X)\) be the unique output of Gram-Schmidt factorization with \(X = QR\). The following statements are true:
    \begin{enumerate}[(i)]
        \item The matrices \(Q\) and \(R\) are statistically independent. \label{thm: characterization of spherical model pt 1}
        \item The law of \(Q\) follows from the uniform distribution on \(O(n, k)\), i.e. \(Q\) follows the law of the first \(k\) columns of a matrix drawn from the Haar measure of the orthogonal group \(O(n)\). \label{thm: characterization of spherical model pt 2}
        \item For \(n \to \infty\), the law of \(R\) concentrates to the delta measure on an upper-triangular matrix \(R^{\star} \in \R^{k \times k}\). Moreover, for any \(b \in (0, 1/2)\), one has
        \begin{equation}\label{eqn: main for R}
            \lim_{n \to \infty} \mathbb{P}\left[ \lVert R - R^{\star} \rVert_{F} < n^{-1/2 + b}\right] = 1,
        \end{equation}
        which shows that \(R\) converges to \(R^{\star}\) at an \(O(n^{-1/2 + o(1)})\) rate. \label{thm: characterization of spherical model pt 3}
        \item Let \(S^{\star}_{\beta}\) be the optimal solution to \Cref{eqn: Goemans-Williamson with logdet}, where the \(\beta\) term coincides with the inverse temperature in \Cref{eqn: spherical model}. The \(R^{\star}\) matrix is the unique right Cholesky factor of \(S^{\star}_{\beta}\) with \(S^{\star}_{\beta} = \left(R^{\star}\right)^{\top}R^{\star}\). \label{thm: characterization of spherical model pt 4}
    \end{enumerate}
\end{thm}

\paragraph{Free energy when \(n \to \infty\)}
Our second main statement shows that the free energy of the \(n\)-vector model at \(n \to \infty\) is exactly solvable by the optimization task in \Cref{eqn: Goemans-Williamson with logdet}.

\begin{thm}\label{thm: characterization of free energy}
    Let \(Z_{n}(\beta)\) be the partition function for \(p\) in \Cref{eqn: spherical model}, defined as follows
    \[
    Z_{n}(\beta) = \int_{x_1 \in \R^{n}}\, dx_1 \ldots \int_{x_k \in \R^{n}} \, dx_k\exp{\left(\beta n\sum_{i,j = 1}^{k}\left<x_{i}, x_{j}\right>A_{ij}\right)}\prod_{i=1}^{k}\delta(1 - \lVert x_{i}\rVert^2).
    \]
    We define the normalized free energy \(Q_n(\beta)\) to be \(Q_{n}(\beta) = \ln\left(
    \frac{Z_{n}(\beta)}{Z_{n}(0)}\right)\). Let \(q^{\star}_{\beta}\) be the maximal value to the optimization task in \Cref{eqn: Goemans-Williamson with logdet}. Then \(Q_n(\beta)\) can be approximated by \(q^{\star}_{\beta}\) by the following equation:
    \[
    Q_n(\beta) = nq^{\star}_{\beta} + O(1),
    \]
    where the \(O(1)\) term is independent of \(n\) and only depends on \(A, \beta\).
    
\end{thm}

\paragraph{Correlation function when \(n \to \infty\)}

The third main statement concerns the correlation function \(\xi\) of the \(n\)-vector model. We write \(S = X^{\top}X\) as a random matrix recording the site-wise correlations. For site \(i\) and site \(j\), the term \( \mathbb{E}\left[S_{ij}\right] = \mathbb{E}\left[\sum_{a = 1 }^{n}(X_i)_{a}(X_j)_{a}\right]\) is the correlation between site \(i\) and site \(j\), and \(\xi\) is defined by \(\xi(i,j) = \mathbb{E}\left[S_{ij}\right]\). The following statement shows that \(S\) concentrates on the maximizer of the optimization task in \Cref{eqn: Goemans-Williamson with logdet}. As a result of the statement, the correlation function \(\xi\) converges and is exactly solvable in the \(n\to \infty\) limit.

\begin{thm}\label{thm: characterization of correlation function}
    Let \(X \sim p\) for \(p\) in \Cref{eqn: spherical model}, and let \(S = X^{\top}X\).
    For \(n \to \infty\), the law of \(S\) concentrates to the delta measure on a positive definite matrix \(S_{\beta}^{\star} \in \R^{k \times k}\) coinciding with the unique maximizer to \Cref{eqn: Goemans-Williamson with logdet}. Moreover, for any \(b \in (0, 1/2)\), one has
    \begin{equation}\label{eqn: concentration with rate WTS}
        \lim_{n \to \infty} \mathbb{P}\left[ \lVert S - S^{\star}_{\beta} \rVert_{F} < n^{-1/2 + b}\right] = 1.
    \end{equation}
    Additionally, one has
    \begin{equation}\label{eqn: concentration with rate WTS 2}
       \lim_{\beta \to \infty} \mathrm{tr}(AS_{\beta}^{\star})= \max_{S \in \R^{k \times k},  S \succeq 0, S_{ii} = 1, \,\, i = 1, \ldots,k} 
       \quad \mathrm{tr}(AS),
    \end{equation}
    which shows that \(S_{\beta}^{\star}\) converges to a maximizer of the semidefinite relaxation of the weighted max-cut problem with edge weight \(A\).
\end{thm}

\subsection{Outline}
This work is organized as follows. \Cref{sec: background} gives preliminary statements on rotation-invariant spherical spin distributions. \Cref{sec: infinite case} proves \Cref{thm: characterization of free energy} and \Cref{thm: characterization of correlation function}. \Cref{sec: finite case} discusses numerically sampling from the \(n\)-vector model and proves \Cref{thm: characterization of spherical model}. \Cref{sec: conclusion} gives concluding remarks and discusses future directions.

\section{Background on \texorpdfstring{\(O(n)\)}{O(n)}-invariant distributions}\label{sec: background}
This section gives the background on \(O(n)\)-invariant models. Throughout this section, we assume \(n > k\). We analyze the distribution of \(S = X^{\top}X\) for \(X \sim p\). One of the difficulties in the analysis on the \(n\)-vector model \(p\) is that \(p\) is supported on \(\left(\mathbb{S}^{n-1}\right)^{k}\), which is a singular measure in \(\left(\R^n\right)^{k}\). We show that the distribution of \(S = X^{\top}X\) has an analytic formula. We show in particular that the strictly lower-triangular part of \(S \in \R^{k \times k}\) has a probability density over an open subset in \(\R^{k(k-1)/2}\). We then give preliminary results on the probability distribution of \(Q, R\) for \((Q, R) = \Phi(X)\).

\paragraph{Distribution of \texorpdfstring{\(S = X^{\top}X\)}{S = X'X}}
We derive the analytic formula for the distribution of \(S = X^{\top}X\). We use \(\splus\) to denote the space of \(k \times k\) positive semidefinite matrices. One sees that \(\splus\) is the support of \(S\). We give a measure to \(\splus\). Define \(\iota \colon \R^{k \times k} \to \R^{k(k+1)/2}\) as the invertible map from a \(k \times k\) symmetric matrix to its upper-triangular entries, i.e.
\begin{equation*}
    \iota\left( \left[M_{ij}\right]_{i,j=1}^{k} \right) = \left( M_{ij}\right)_{1 \leq i \leq j \leq k}.
\end{equation*}
Note that \(\iota(\splus)\) is a closed subset of \(\R^{k(k+1)/2}\) with a nonempty interior. Let \(\mu_{\R^{k(k+1)/2}}\) be the Lebesgue measure on \(\R^{k(k+1)/2}\) and let \(\mu_{\iota(\splus)}\) be the subspace measure of \(\mu_{\R^{k(k+1)/2}}\) restricted to \(\iota(\splus)\). The measure we give to \(\splus\) is \(\mu_{S} = (\iota^{-1})_{\#}\left(\mu_{\iota(\splus)}\right)\), i.e. \(\mu_{S}\) is the pushforward of the Lebesgue measure on \(\splus\) by \(\iota^{-1}\). 

The goal of this subsection is to prove the following statement, which characterizes the distribution function of \(S\).
\begin{thm}\label{theorem: main for S}
Let \(p\) be the \(n\)-vector model in \Cref{eqn: spherical model}. For an \(n \times k\) matrix \(X = \left[x_1, \ldots, x_k\right]\), write \(p(X) = p(x_1, \ldots, x_k)\). For a continuous function \(f \colon \R^{k \times k} \to \R\), one has
\begin{equation}\label{eqn: change of variable for S}
\int_{X \in \R^{n \times k}}dX\, p(X) f(X^{\top}X) = \int_{S \in \splus} \mu_{\mathcal{S}}(dS) p_{S}(S) f(S),
\end{equation}
where for the constant
\footnote{
The \(\omega(n,k)\) factor is calculated by
\[
c(n,p) = \pi^{(p^2-p)/4}2^{np/2}\prod_{j=1}^{p}\Gamma\left(\frac{n-j+1}{2}\right), \quad \omega(n,p) = 1/c(n,p).
\]
}
\(C_{n,k} = \left(\sqrt{2\pi}\right)^{nk}\omega(n, k)\), one has
\begin{equation}\label{eqn: measure for S}
    p_{S}(S) = 2^{k}C_{n,k}\exp\left(\beta n\mathrm{tr}(SA)\right)
    \det(S)^{(n-k-1)/2}\prod_{i= 1}^{k}\delta(1 - S_{ii} ).
\end{equation}

Let \(X \sim p\) and let \(L\) be the strictly lower-triangular part of \(X^{\top}X\). Write \(S(L) = I_{k} + L + L^{\top}\). The support of \(L\) is \(\{L \in \R^{k(k-1)/2} \mid S(L) \in \splus\}\). As a consequence of \Cref{eqn: change of variable for S} and \Cref{eqn: measure for S}, the distribution function of \(L\) has an (unnormalized) density function with respect to the Lebesgue measure on its support and can be written as follows:
\[
p_{L}(L) = 2^{k}C_{n,k}\exp\left(\beta n\mathrm{tr}(S(L)A)\right)\det(S(L))^{(n-k-1)/2}.
\]
\end{thm}

Therefore, \Cref{theorem: main for S} shows that \(S = X^{\top}X\) has a nice distribution function. Moreover, the distribution of the strictly lower-triangular part of \(S\) has a probability density. In particular, the \(\exp\left(\beta n\mathrm{tr}(S(L)A)\right)\) factor of \(p_{L}\) suggests that the spin dimension \(n\) has a similar effect as the inverse temperature \(\beta\). The temperature-like effect of \(n\) is key to deriving the main statements stated in \Cref{sec: intro}. 

We give two intermediate results necessary for proving \Cref{theorem: main for S}. The first intermediate result is a statement regarding \(O(n)\)-invariant distributions which are absolutely continuous with respect to the Lebesgue measure. The result is proven in \cite{eaton1983multivariate} and we quote it here.
\begin{prop}\label{prop: density on homogeneous spaces}
(Proposition 7.6 of \cite{eaton1983multivariate})
Let \(X\in \R^{n \times k}\) be a random matrix with distribution \(X \sim \nu \). Suppose that \(\nu\) is \(O(n)\)-invariant. 
That is, for any Borel subset \(B \subset \R^{n \times k}\), and any \(\Gamma \in O(n)\), we assume that one has
\[
\nu(B) = \nu(\Gamma B).
\]

Suppose the law of \(X\) has a density function \(g\) with respect to the Lebesgue measure on \(\R^{n \times k}\) and that there exists \(h\) so that \(g(X) = h(X^{\top}X)\). Then \(S = X^{\top}X\) has the following density \(g_S\) with respect to \(\mu_{\mathcal{S}}\):
\[
g_S(S) 
= C_{n,k}\det(S)^{(n-k-1)/2}h(S).
\]
In particular, for a continuous function \(f \colon \R^{k \times k} \to \R\), one has
\begin{equation*}
\int_{X \in \R^{n \times k}}\nu(dX)\, f(X^{\top}X) = \int_{S \in \splus} \mu_{\mathcal{S}}(dS) g_{S}(S) f(S).
\end{equation*}
\end{prop}

The second intermediate result is a statement that formulates the uniform measure on \(\mathbb{S}^{n-1}\) as the limit of characteristic functions. We state it and we prove it after proving \Cref{theorem: main for S}.
\begin{prop}\label{prop: sphere dirac delta}
Let \(B(x,r) \subset \R^{n}\) denote the ball of radius \(r\) centered at \(x\). Suppose \(q\colon B(0,2) \to \R\) is continuous. Then
\begin{equation}\label{eqn: sphere dirac delta WTS}
    \int_{x \in \R^{n}} dx q(x)\delta(1 - \lVert x \rVert_{2})  = \lim_{t \to 0}\frac{1}{t}\int_{x \in \R^{n}}dxq(x) \chi\left(\lVert x \rVert \in [1, 1+t]\right),
\end{equation}
where \(\chi\) is the characteristic function, i.e. 

\[\chi\left(\lVert x \rVert \in [1, 1+t]\right) = 
\begin{cases}
1 & \text{if \(\lVert x \rVert \in [1, 1+t]\)},\\
0 &\text{otherwise}.
\end{cases}
\]

Moreover, for any \(c > 0\), \(\varepsilon > 0\) and any continuous univariate function \(b \colon (1-\varepsilon, 1+\varepsilon) \to \R\), one has
\begin{equation}\label{eqn: sphere dirac delta WTS 2}
    \int_{1-\varepsilon}^{1+\varepsilon} dx \,b(x)\delta(1 - x)  = \lim_{t \to 0}
    \frac{1}{t}\int_{1 - \min(t, \varepsilon)}^{1 + \min(t, \varepsilon)}dx\,b(x) \chi\left(x \in [1, 1+t + ct^2]\right).
\end{equation}
\end{prop}

We first prove \Cref{theorem: main for S} using the two intermediate results.
\begin{proof}
(Proof of \Cref{theorem: main for S})

Throughout this proof, we assume that \(X = \left[x_1, \ldots, x_k\right]\).
As \(f\) is continuous, one can use \Cref{prop: sphere dirac delta} and write the left hand side of \Cref{eqn: change of variable for S} as follows
\begin{align*}
&\int_{X \in \R^{n \times k}}dX\, p(X) f(X^{\top}X)
\\
=&\lim_{t_1, \ldots, t_{k} \to 0}\int_{X \in \R^{n \times k}} dX \exp(\beta n \mathrm{tr}(AX^{\top}X))f(X^{\top}X)\prod_{i=1}^{k}
\frac{1}{t_i}\chi{\left(\lVert x_i \rVert \in [1, 1 + t_i]\right)}.\\
\end{align*}

Then, to use \Cref{prop: density on homogeneous spaces}, we construct \(\nu\) to be a distribution with density \(g(X) = \exp(\beta n \mathrm{tr}(AX^{\top}X))\prod_{i = 1}^{k}\psi(\lVert x_i \rVert) \). The function \(\psi\) is a smooth non-negative function satisfying \(\psi(t) = 1\) for \(t \in [0, 2]\) and \(\psi(t) = 0\) for \(t \geq 3\). Then \(\nu\) is an \(O(n)\)-invariant compactly supported distribution whose density is \(\exp(\beta n \mathrm{tr}(AX^{\top}X))\) for the region \(\left\{X \mid \lVert x_i \rVert \leq 2 \text{ for } i = 1, \ldots, k \right\}\).
By the construction of \(\nu\), we have
\begin{align*}
&\int_{X \in \R^{n \times k}}dX\, p(X) f(X^{\top}X)\\
=&\lim_{t_1, \ldots, t_{k} \to 0}\int_{X \in \R^{n \times k}} \nu(dX)f(X^{\top}X)\prod_{i=1}^{k}
\frac{1}{t_i}\chi{\left(\lVert x_i \rVert \in [1, 1 + t_i]\right)}.\\
=&\lim_{t_1, \ldots, t_{k} \to 0}\int_{S \in \splus} \mu_{S}(dS)C_{n,k}\det(S)^{(n-k-1)/2}\exp(\beta n \mathrm{tr}(AS))f(S)\prod_{i=1}^{k}
\frac{1}{t_i}\chi{\left(\lvert S_{ii} \rvert \in [1, (1 + t_i)^2]\right)},
\end{align*}
where the last equality holds because \(\lVert x_i\rVert^2 = S_{ii}\) when \(S = X^{\top}X\) and \(X = \left[x_1, \ldots, x_{k}\right]\).

We then exchange the order of limit and integration. Applying \Cref{prop: sphere dirac delta} again, one has 
\begin{align*}
&\int_{X \in \R^{n \times k}}dX\, p(X) f(X^{\top}X)\\
    =
    &\int_{S \in \splus}\mu_{S}(dS)
    C_{n,k}\exp(\beta n \mathrm{tr}(AS))f(S) \det(S)^{\frac{n-k-1}{2}}
    \left(\lim_{t_1, \ldots, t_{k} \to 0}\prod_{i=1}^{k}
\frac{1}{t_i}\chi{\left(S_{ii} \in [1, 1 + 2t_i + t_i^2]\right)}\right)\\
    =
    &2^{k}\int_{S \in \splus}\mu_{S}(dS) C_{n,k}
    \exp(\beta n \mathrm{tr}(AS))f(S) \det(S)^{\frac{n-k-1}{2}}\prod_{i= 1}^{k}\delta(1 - S_{ii}),
\end{align*}
where the last equality holds due to Fubini's theorem and \Cref{eqn: sphere dirac delta WTS 2}. Therefore, we have proven \Cref{eqn: change of variable for S} with \(p_S\) satisfying \Cref{eqn: measure for S}. Our statements for \(L\) are direct consequences of \Cref{eqn: change of variable for S}. Thus, we are done.
\end{proof}

It remains to prove \Cref{prop: sphere dirac delta}.
\begin{proof}
(Proof of \Cref{prop: sphere dirac delta})

We note that \Cref{eqn: sphere dirac delta WTS} directly follows from the coarea formula, and a detailed proof of \Cref{eqn: sphere dirac delta WTS} can be found in Chapter 3 of \cite{evans2018measure}.

It remains to show \Cref{eqn: sphere dirac delta WTS 2}. 
We assume for the rest of the proof that \(b\) is non-negative, and we note that the general case comes from the linearity of the integral. By non-negativity of \(b\), one has
\begin{equation*}
\frac{1}{t}\int_{1 - \min(t, \varepsilon)}^{1 + \min(t, \varepsilon)}dx\,b(x) \chi\left(\lVert x \rVert \in [1, 1+t + ct^2]\right) \geq \frac{1}{t}\int_{1 - \min(t, \varepsilon)}^{1 + \min(t, \varepsilon)}dx\,b(x) \chi\left(\lVert x \rVert \in [1, 1+t ]\right).
\end{equation*}
Then, taking the limit of \(t \to 0\), one has
\[
    \lim_{t \to 0}
    \frac{1}{t}\int_{1 - \min(t, \varepsilon)}^{1 + \min(t, \varepsilon)}dx\,b(x) \chi\left(x \in [1, 1+t + ct^2]\right) \geq b(1).
\]
On the other hand, let \(\eta > 0\) be any constant. When \(t < \eta/c\), one has \(ct^2 < \eta t\), and so the following bound holds:
\begin{equation*}
\frac{1}{t}\int_{1 - \min(t, \varepsilon)}^{1 + \min(t, \varepsilon)}dx\,b(x) \chi\left(\lVert x \rVert \in [1, 1+t + ct^2]\right) \leq \frac{1}{t}\int_{1 - \min(t, \varepsilon)}^{1 + \min(t, \varepsilon)}dx\,b(x) \chi\left(\lVert x \rVert \in [1, 1 + (1 + \eta)t ]\right).
\end{equation*}
Taking the limit of \(t \to 0\), one has 
\[
\lim_{t \to 0}
    \frac{1}{t}\int_{1 - \min(t, \varepsilon)}^{1 + \min(t, \varepsilon)}dx\,b(x) \chi\left(x \in [1, 1+t + ct^2]\right) \leq (1+\eta)b(1).
\]
Note that \(\eta\) can be arbitrarily close to \(0\), and so we have proven
\[
b(1) = \frac{1}{t}\int_{1-\varepsilon}^{1+\varepsilon}dx\,b(x) \chi\left(\lVert x \rVert \in [1, 1+t + ct^2]\right).
\]
Note that \(b(1) = \int_{1-\varepsilon}^{1+\varepsilon} dx \,b(x)\delta(1 - x)\), and so we are done.
\end{proof}

\paragraph{Distribution of \texorpdfstring{\((Q, R) = \Phi(X)\)}{(Q,R) = Phi(X)}}

Let \(O(n, k)\) denote the space of \(n \times k\) matrices \(Q\) satisfying \(Q^{\top}Q = I_k\).
Let \(G_{U}^{+}\) denote the space of \(k \times k\) upper-triangular matrices with all diagonal entries being positive. For an invertible matrix \(M \in \R^{n \times k}\), let \(\Phi(M) = (Q, R)\) be the unique output of the Gram-Schmidt algorithm with \(U \in O(n, k)\) and \(R \in G_{U}^{+}\).
We recall a basic statement on \(O(n)\)-invariant distributions:
\begin{prop}\label{prop: distribution on homogeneous spaces}
(Proposition 7.3 in \cite{eaton1983multivariate})
Let \(X\in \R^{n \times k}\) be a random matrix with distribution \(X \sim \nu\) so that \(X\) is almost surely of full rank. Let \((Q, R) = \Phi(X)\) be the unique output of Gram-Schmidt. Suppose that \(\nu\) is \(O(n)\)-invariant as defined in \Cref{prop: density on homogeneous spaces}.
Then the following statements are true:
\begin{enumerate}
    \item \(Q\) and \(R\) are statistically independent.
    \item The law of \(Q\) is a uniform distribution on \(O(n, k)\), i.e. \(Q\) follows the law of the first \(k\) columns of a matrix drawn from the Haar measure of \(O(n)\).
\end{enumerate}
\end{prop}

For the \(n\)-vector model \(X \sim p\), the \(O(n)\) symmetry holds through the functional form. From \Cref{theorem: main for S}, we see that \(S = X^{\top}X\) is generically invertible for \(X \sim p\), which shows that \(X\) is also generically of full rank. Therefore, the results in \Cref{prop: distribution on homogeneous spaces} hold for any \(\beta\) and \(A\). In particular, \Cref{prop: distribution on homogeneous spaces} shows that (\ref{thm: characterization of spherical model pt 1})-(\ref{thm: characterization of spherical model pt 2}) in \Cref{thm: characterization of spherical model} is a simple consequence of \(O(n)\) invariance of the \(n\)-vector model.

\section{Free energy and correlation function at \texorpdfstring{\(n \to \infty\)}{n to infinity}}\label{sec: infinite case}
This section calculates the free energy and correlation function for the \(n\)-vector model and shows that they are as stated in \Cref{thm: characterization of free energy} and \Cref{thm: characterization of correlation function}. The main idea of the proof strategy is that the spin dimension \(n\) in the \(n\)-vector model is temperature-like. Following the observation, one can prove the desired limiting behavior with the Laplace method. 

We use \(\sGW \subset \R^{k \times k}\) to denote the set of positive semidefinite matrices with all diagonal entries equal to one.
In both results, it is implied that the regularized SDP task in \Cref{eqn: Goemans-Williamson with logdet} admits a unique maximizer. We prove the uniqueness for completeness.
\begin{prop}\label{prop: uniqueness of f_beta maximizer}
 The optimization task
\begin{equation}\label{eqn: uniqueness of f_beta maximizer WTS}
       {\mathrm{max}}_{S \in \sGW} \beta\mathrm{tr}(AS) +  \frac{1}{2}\logdet(S)
\end{equation}
has a unique maximizer \(S^{\star}_{\beta}\).
\end{prop}

\begin{proof}
Note that \(\sGW\) is an infinite intersection of closed sets and is thus closed. Let \(e_i\) denote the \(i\)-th standard basis vector in \(\R^{k}\). Boundedness of \(\sGW\) follows from the fact that each diagonal entry of \(S \in \sGW\) is one, and thus for \(v = e_i - e_j\) one has
\[
v^{\top}Sv = -2S_{ij} + 2 \geq 0.
\]
Similarly, taking \(v = e_i + e_j\) shows \(2S_{ij} + 2 \geq 0\). Therefore, all entries of \(S\) are bounded. Thus \(\sGW\) is compact.

Let \(\sigma_{\mathrm{min}}(S)\) denote the minimal singular value of \(S\). Due to the log-determinant term in \Cref{eqn: uniqueness of f_beta maximizer WTS}, there exists \(c > 0\) for which the maximum is obtained at \(\sGW \cap \{S \mid \sigma_{\mathrm{min}}(S) \geq c\} \). Moreover, the maximal singular value of \(S\) is bounded from above by \(k\). Therefore, it is equivalent to take the optimization to be over \(\sGW\cap \{S \mid \sigma_{\mathrm{min}}(S) \geq c\}\), on which the objective function of \Cref{eqn: uniqueness of f_beta maximizer WTS} is smooth and strictly concave. The existence of \(S^{\star}_{\beta}\) comes from the compactness of \(\sGW\cap \{S \mid \sigma_{\mathrm{min}}(S) \geq c\}\), and the uniqueness comes from the strict concavity of the objective function.
\end{proof}

The next statement is a formulation of the Laplace method, which we shall use for proving \Cref{thm: characterization of free energy} and \Cref{thm: characterization of correlation function}.
\begin{prop}\label{prop: laplace method}
Let \(P\) be a compact subset of \(\R^{d}\) with a nonempty interior and let \(\mu_{P}\) be the restricted Lebesgue measure of \(P\) as a subset of \(\R^{d}\).
Let \(f \colon P \to \R\) be a smooth and strictly concave function with a unique maximizer \(x^{\star}\) in the interior of \(P\) and we assume that \(f\) is strictly concave at a neighborhood of \(x^{\star}\). Let \(X\) be a random variable supported on \(P\) whose density with respect to \(\mu_{P}\) is given by \(p_{n}(x) = \exp{(nf(x))}g(x)\), where \(g \colon P \to \R_{\geq 0}\) is smooth, bounded with \(g(x^{\star}) \not = 0\).

Then, for any \(b \in (0,1/2)\), one has
\begin{equation}\label{eqn: prop concentration WTS 1}
    \lim_{n \to \infty} \mathbb{P}_{X \sim p_{n}} \left[ \lVert X - x^{\star} \rVert_{F} < n^{-1/2 + b}\right] = 1.
\end{equation}
Write \(f^{\star} = f(x^{\star})\). For sufficiently large \(n\), one has
\begin{equation}\label{eqn: prop concentration WTS 2}
     \ln{\left(\int_{P}p_{n}(x) dx\right)} = nf^{\star} -d/2\ln{n} + O(1),
\end{equation}
where \(O(1)\) depends on \(f, g, P\) but does not depend on \(n\).
\end{prop}

We defer the proof of \Cref{prop: laplace method} to the end of this section.
We first prove \Cref{thm: characterization of free energy} assuming the Laplace method calculations.

\begin{proof}
(Proof of \Cref{thm: characterization of free energy})

We recall the definition of \(L\) in \Cref{theorem: main for S} and the operation \(S(L) = I_{k} + L + L^{\top}\). In particular, \(L\) is supported on \(\mathcal{L} = \{L \mid S(L) \in \splus\}\) with the following density with respect to the Lebesgue measure on \(\mathcal{L}\)
\[
p_{L}(L) := 2^{k}C_{n,k}\exp\left(\beta n\mathrm{tr}(S(L)A)\right)\det(S(L))^{(n-k-1)/2}.
\]
Moreover, one sees that \(\mathcal{L}\) is compact. To use \Cref{prop: laplace method}, we define
\[ L^{\star}_{\beta} =   {\mathrm{argmax}}_{L \in \mathcal{L}}\,\beta\mathrm{tr}(AS(L)) +  \frac{1}{2}\logdet(S(L)).
\]

As the mapping \(L \to S(L)\) a bijection from \(\mathcal{L}\) to \(\sGW\), one has \(S^{\star}_{\beta} = S(L^{\star}_{\beta})\). In other words, \(L^{\star}_{\beta}\) is the strictly lower-triangular part of \(S^{\star}_{\beta}\). From \Cref{prop: uniqueness of f_beta maximizer}, we know that \(S^{\star}_{\beta}\) is positive definite, and so \(L^{\star}_{\beta}\) lies in the interior of \(\mathcal{L}\). 

We write \(d := k(k-1)/2\). We emphasize the dependency of \(p_{L}\) on \(n\) by writing
\begin{equation}\label{eqn: def of p_n for L}
    p_{n}(L) = 2^{k}C_{n,k}\exp\left((n-k-1)f(L)\right)g(L),
\end{equation}
where \(g(L) := \exp\left((k+1)\beta\mathrm{tr}(S(L)A)\right)\), and \(f(L) := \beta\mathrm{tr}(AS(L)) + \frac{1}{2}\logdet(S(L))\). By construction one has \(L^{\star}_{\beta} = \argmax_{L \in \mathcal{L}}f(L)\). One sees that \(g(L^{\star}_{\beta}) > 0\). Therefore, all the assumptions of \Cref{prop: laplace method} hold, and we can apply the results to \(p_n\). Importantly, \(f(\mathcal{L}^{\star}_{\beta}) = q_{\beta}^{\star}\).

We use \(\mu_{L}\) to denote the Lebesgue measure on \(\mathcal{L}\). By \Cref{theorem: main for S}, one has
\[
\begin{aligned}
    Z_{n}(\beta) = &\int_{X \in \R^{n \times k}}dX\, p(X)\\
    = &\int_{S \in \splus} \mu_{\mathcal{S}}(dS) 2^{k}C_{n,k}\exp\left(\beta n\mathrm{tr}(SA)\right)
    \det(S)^{(n-k-1)/2}\prod_{i= 1}^{k}\delta(1 - S_{ii} ) 
    \\
    = &\int_{L \in \mathcal{L}} \mu_{L}(dL)\, 2^{k}C_{n,k}\exp\left(\beta n\mathrm{tr}(S(L)A)\right)\det(S(L))^{(n-k-1)/2}\\
    =&\int_{L \in \mathcal{L}} \mu_{L}(dL)p_n(L)\\
    =&\int_{L \in \mathcal{L}} \mu_{L}(dL)\exp\left((n-k-1)f(L)\right)g(L).
\end{aligned}
\]
Therefore, \Cref{eqn: prop concentration WTS 2} in \Cref{prop: laplace method} applies. For sufficiently large \(n\), one has
\begin{align*}
    &\int_{L \in \mathcal{L}} \mu_{L}(dL)\exp\left((n-k-1)f(L)\right)g(L)\\
    = &nf(L^{\star}_{\beta}) - k(k-1)/4\ln{n} + O(1)\\
    = &nq_{\beta}^{\star} - k(k-1)/4\ln{n} + O(1).
\end{align*}

Thus \(Q_{n}(\beta) = \ln{Z_{n}(\beta)} - \ln{Z_{n}(0)} = n q_{\beta}^{\star} - n q_{\beta = 0}^{\star} + O(1)\). 
For \(\beta = 0\), one has \(f(L) = \frac{1}{2}\logdet(S(L))\). By applying Cauchy-Schwarz inequality on the fact that \(\mathrm{tr}(S) = k\) for \(S \in \sGW\), one has \(q_{\beta = 0}^{\star} = \frac{1}{2}\logdet(I_k) = 0\). Therefore \[Q_{n}(\beta) = \ln{Z_{n}(\beta)} - \ln{Z_{n}(0)}= n q_{\beta}^{\star} + O(1),\]
and so we are done.
\end{proof}

We now prove \Cref{thm: characterization of correlation function}.
\begin{proof}
(Proof of \Cref{thm: characterization of correlation function})

To prove \Cref{eqn: concentration with rate WTS}, we note that \( \lVert S^{\star}_{\beta} - S(L)\rVert = 2 \lVert L^{\star}_{\beta} - L\rVert\). We see from the proof of \Cref{thm: characterization of free energy} that \Cref{prop: laplace method} applies to the distribution \(L \sim p_n\) as defined in \Cref{eqn: def of p_n for L}. Thus, for any \(b \in (0,1/2)\), one has
\[\lim_{n \to \infty} \mathbb{P}_{L \sim p_{n}} \left[ \lVert L - L^{\star}_{\beta} \rVert_{F} < (n-k-1)^{-1/2 + b}\right] = 1.\]

Noting that for \(n > 2k + 2\) one has \(n-k-1 > n/2\), and the following holds
\[(n-k-1)^{-1/2 + b} < 2n^{-1/2 + b}.\]

Combined with \( \lVert S^{\star}_{\beta} - S(L)\rVert = 2 \lVert L^{\star}_{\beta} - L\rVert\), we obtain
\begin{equation}\label{eqn: concentration with rate main step}
    \lim_{n \to \infty} \mathbb{P}_{S \sim p_{S}} \left[ \lVert S - S^{\star}_{\beta} \rVert_{F} < 4n^{-1/2 + b}\right] = 1.
\end{equation}

As \(4n^{-1/2 + b/2} < n^{-1/2 + b}\) for \(n\) sufficiently large, we see that \Cref{eqn: concentration with rate WTS} in \Cref{thm: characterization of correlation function} is implied by \Cref{eqn: concentration with rate main step}.

It remains to prove \Cref{eqn: concentration with rate WTS 2} in \Cref{thm: characterization of correlation function}. For the remainder of this proof, our use of the big \(O\) notation only includes the dependence on \(\beta\).
Let \(\sstargw\) be a (not necessarily unique) maximizer to the SDP \(\mathrm{max}_{S \in \sGW}\mathrm{tr}(AS)\). We write \(\fstargw = \mathrm{tr}(A\sstargw)\). The proof proceeds by considering a perturbation to \(\sstargw\) of the form \(S(\alpha) := \alpha\sstargw + (1-\alpha)I_{k}\). One can directly check that \(S(\alpha)\) is positive definite for \(\alpha \in (0,1)\). By the optimality of \(S^{\star}_{\beta}\) in \Cref{eqn: uniqueness of f_beta maximizer WTS}, one has
\begin{equation}\label{eqn: theorem Sbeta step 1}
\begin{aligned}
    &\mathrm{tr}(AS^{\star}_{\beta}) +  \frac{1}{2\beta}\logdet(S^{\star}_{\beta}) \\\geq 
&\fstargw\alpha + \mathrm{tr}(A)(1-\alpha) + \frac{1}{2\beta}\logdet(\alpha\sstargw + (1-\alpha)I_{k}).
\end{aligned}
\end{equation}

For the left hand side of \Cref{eqn: theorem Sbeta step 1}, one has \(\logdet(S^{\star}_{\beta}) \leq \mathrm{tr}(S^{\star}_{\beta} - I_k) = 0\). The right hand side of \Cref{eqn: theorem Sbeta step 1} satisfies \(\logdet(\alpha\sstargw + (1-\alpha)I_{k}) \geq k\ln{\left(1-\alpha\right)}\). As a consequence, \Cref{eqn: theorem Sbeta step 1} implies
\begin{equation}\label{eqn: theorem Sbeta step 2}
\begin{aligned}
    \mathrm{tr}(AS^{\star}_{\beta}) \geq 
\fstargw + (\mathrm{tr}(A) - \fstargw)(1-\alpha) + \frac{1}{2\beta}k\ln{\left(1-\alpha\right)}.
\end{aligned}
\end{equation}
In particular, one can plug in \(1-\alpha =  \frac{1}{\beta}\). Under this choice, one has \((\mathrm{tr}(A) - \fstargw)(1-\alpha) = O(\beta^{-1})\) and \(\frac{1}{2\beta}k\ln{\left(1-\alpha\right)} = O(\beta^{-1}\ln{\beta})\). Therefore, choosing \(\alpha = 1 - 1/\beta\) in \Cref{eqn: theorem Sbeta step 2} implies
\[
\mathrm{tr}(AS^{\star}_{\beta}) \geq \fstargw + O( \beta^{-1}\ln{\beta}),
\]
and thus we are done after taking \({\beta \to \infty}\).
\end{proof}

It remains to prove \Cref{prop: laplace method}.
\begin{proof}
(Proof of \Cref{prop: laplace method})

We use \(B(x, r)\) to denote the ball of radius \(r\) centered at \(x\). The big \(O\) notation in this proof only includes the dependency on \(n\). For proving \Cref{eqn: prop concentration WTS 1}, we claim that there exists constants \(\delta,m,M,\eta > 0\), all of which independent of \(n\), such that the following holds for \(\eps \in (0, \delta)\):
\begin{equation}\label{eq: thm concentration step 1}
\ln{\left(\int_{P \cap B(x^{\star}, \eps)}p_{n}(x) dx\right)} \geq nf^{\star}+ \ln{\gamma(d/2, nM\eps^2)} +  O(\ln(n)),
\end{equation}
and 
\begin{equation}\label{eq: thm concentration step 2}
\ln{\left(\int_{P - B(x^{\star}, \eps)}p_{n}(x) dx\right)} \leq  nf^{\star} - n\min{\left(\eta ,m\eps^2\right)} + O(1),
\end{equation}
where \(\gamma\) is the lower incomplete gamma function \cite{abramowitz1988handbook}.
Therefore, taking \(\eps = n^{-1/2 + b}\) for any \(b \in (0,1/2)\) implies that \(n\eps^2 \to \infty\), which in turn implies \[\ln{\gamma(d/2, nM\eps^2)} \to \ln{\Gamma(d/2)} = O(1).\]

We show why \Cref{eq: thm concentration step 1} and \Cref{eq: thm concentration step 2} imply \Cref{eqn: prop concentration WTS 1}. Taking \(\eps = n^{-1/2 + b}\), one sees that the right-hand side of \Cref{eq: thm concentration step 1} is much larger than that of \Cref{eq: thm concentration step 2}. Thus one has
\[
\lim_{n \to \infty}\frac{
\int_{P \cap B(x^{\star}, n^{1/2 - b})}p_{n}(x) dx
}{
\int_{P}p_{n}(x) dx
} = 1,
\]
which implies \Cref{eqn: prop concentration WTS 1}.

We then prove \Cref{eq: thm concentration step 1} and \Cref{eq: thm concentration step 2}.
Due to the smoothness and strict concavity of \(f\) around \(x^{\star}\), the point \(x^{\star}\) satisfies \(\nabla f(x^{\star}) = 0\) and there exist positive constants \(m, M\) so that \(-M I_{d}\preceq \nabla^{2}f(x^{\star}) \preceq -4m I_{d} \). Moreover, as \(f\) is smooth, it follows that there exists a sufficiently small radius \(\delta > 0\) so that \(B(x^{\star}, \delta) \subset P\), and \(x \in B(x^{\star}, \delta)\) satisfies \(-2MI_{d}\prec \nabla^{2}f(x) \prec -2mI_{d}\). The following holds by using the Taylor remainder theorem for \(x \in B(x^{\star}, \delta)\):
\begin{equation}
   f(x^{\star}) - M\lVert x^{\star} - x \rVert^2 < f(x) < f(x^{\star}) - m \lVert x^{\star} - x \rVert^2.
\end{equation}

By a similar continuity argument, by possibly shrinking \(\delta\), one can assume \(x \in B(x^{\star}, \delta)\) implies \( \frac{1}{2} g(x^{\star}) \leq g(x) \leq 2g(x^{\star})\). Moreover, as the maximizer is unique, one can further shrink \(\delta\) so that there exists \(\eta > 0\) such that \(x \not \in B(x^{\star}, \delta)\) implies \(f(x) < f(x^{\star}) - \eta\). By assumption, \(g\) is bounded on \(P\) and we let \(\gmax\) denote its supremum, i.e. \(\gmax = \sup_{x \in P}g(x)\). 

We prove \Cref{eq: thm concentration step 2}. For \(x \in \R^{d}\) and \(0< l < l'\), define \(A(x,l,l')\) as the annulus centered at \(x\) with radius parameter \((l,l')\), i.e. \(A(x, l, l') := B(x, l') - B(x,l)\). For \(\eps < \delta\), one computes
\begin{align*}
    &\int_{P - B(x^{\star}, \eps)}p_{n}(x) dx \\
    = 
    &\int_{P - B(x^{\star}, \delta)}p_{n}(x) dx +
    \int_{A(x^{\star}, \eps, \delta)}p_{n}(x) dx \\
    \leq
    &\gmax\left(\mu_{P}(P)\exp{\left(n(f^{\star} - \eta)\right)} + \mu_{P}\left(B(x^{\star}, \delta)\right)\exp{\left(n(f^{\star} - m\eps^2)\right)}\right)
    \\
    \leq
    &\gmax\mu_{P}(P)\left(\exp{\left(n(f^{\star} - \eta)\right)} + \exp{\left(n(f^{\star} - m\eps^2)\right)}\right).
\end{align*}
Therefore one has
\begin{equation}
    \int_{P - B(x^{\star}, \eps)}p_{n}(x) dx \leq 2\gmax\mu_{P}(P)\exp{\left(nf^{\star} - n\min{\left(\eta, m \eps^2\right)} \right)},
\end{equation}
which implies \Cref{eq: thm concentration step 2}.

We now prove \Cref{eq: thm concentration step 1}. The proof uses the fact that the chi-squared distribution \(\chi^{2}(d)\) satisfies \(\mathbb{P}_{y \sim \chi^2(d)}\left[ y \in [0, l] \right] = \frac{\gamma(d/2, l/2)}{\Gamma(d/2)}\). For \(\eps < \delta\), direct computation shows
\begin{align*}
    &\int_{B(x^{\star}, \eps)}p_{n}(x) dx\\ \geq 
    &\int_{B(x^{\star}, \eps)}
    \exp{\left(n(f^{\star} - M \lVert x - x^{\star} \rVert_2^2)\right)}g(x) dx \\
    \geq 
    &\frac{1}{2}g(x^{\star})\exp{(nf^{\star})}\int_{B(0, \eps)}
    \exp{\left(-nM \lVert x  \rVert_2^2\right)} dx \\
    =&\frac{1}{2}g(x^{\star})\exp{(nf^{\star})}(2\pi)^{d/2}(\sqrt{2nM})^{-d}\int_{B(0, \sqrt{2nM}\eps)}
   \frac{\exp{\left(-1/2 \lVert x  \rVert_2^2\right)}}{(2\pi)^{d/2}}  dx \\
   =&\frac{1}{2}g(x^{\star})\exp{(nf^{\star})}(2\pi)^{d/2}(\sqrt{2nM})^{-d}\mathbb{P}_{x \sim \mathcal{N}(0, I_d)}\left[ \sum_{i=1}^{d}x_i^2 \in [0, 2nM\eps^2] \right] \\
   =&\frac{1}{2}g(x^{\star})\exp{(nf^{\star})}(2\pi)^{d/2}(\sqrt{2nM})^{-d}\mathbb{P}_{y \sim \chi^2(d)}\left[ y \in [0, 2nM\eps^2] \right] \\
   =&\frac{1}{2}g(x^{\star})\exp{(nf^{\star})}(2\pi)^{d/2}(\sqrt{2nM})^{-d}\frac{\gamma(d/2, nM\eps^2)}{\Gamma(d/2)},
\end{align*}
which proves \Cref{eq: thm concentration step 1} by taking log on both sides.

We then prove \Cref{eqn: prop concentration WTS 2}. Similar to the previous computation, one can upper bound the integral of \(p_n\) within \(\int_{B(x^{\star}, \eps)}\) by
\begin{align*}
    &\int_{B(x^{\star}, \eps)}p_{n}(x) dx\\ \leq 
    &\int_{B(x^{\star}, \eps)}
    \exp{\left(n(f^{\star} - m \lVert x - x^{\star} \rVert_2^2)\right)}g(x) dx \\
    \leq 
    &2g(x^{\star})\exp{(nf^{\star})}\int_{B(0, \eps)}
    \exp{\left(-nm \lVert x  \rVert_2^2\right)} dx \\
   =&2g(x^{\star})\exp{(nf^{\star})}(2\pi)^{d/2}(\sqrt{2nm})^{-d}\frac{\gamma(d/2, nm\eps^2)}{\Gamma(d/2)}.
\end{align*}
Thus, the derived lower and upper bounds imply
\[
\ln{\left(\int_{B(x^{\star}, \eps)}p_{n}(x) dx\right)} \geq nf^{\star} -d/2\ln{n} + \ln{\frac{\gamma(d/2, nM\eps^2)}{\Gamma(d/2)}} + O(1),
\]
and
\[
\ln{\left(\int_{B(x^{\star}, \eps)}p_{n}(x) dx\right)} \leq nf^{\star} -d/2\ln{n} + \ln{\frac{\gamma(d/2, nm\eps^2)}{\Gamma(d/2)}} + O(1).
\]
Again one can take \(\eps = n^{-1/2 + b}\) for \(b \in (0,1/2)\), in which case \(\ln{\frac{\gamma(d/2, nc\eps^2)}{\Gamma(d/2)}} \to 0\) for any positive \(c\). We take sufficiently large \(n\) so that \(n^{-1/2 + b} \leq \delta\), for which the derived bounds can be combined to the estimate
\[
\ln{\left(\int_{B(x^{\star}, \eps)}p_{n}(x) dx\right)} = nf^{\star} -d/2\ln{n} + O(1).
\]

As the contribution from outside \(B(x^{\star}, \eps)\) is asymptotically negligible, for sufficiently large \(n\) one has
\[
\int_{B(x^{\star}, \eps)}p_{n}(x) dx \leq \int_{P}p_{n}(x) dx \leq 2\int_{B(x^{\star}, \eps)}p_{n}(x) dx.
\]
Therefore, one has
\[
\ln{\left(\int_{P}p_{n}(x) dx\right)} = nf^{\star} -d/2\ln{n} + O(1),
\]
which proves \Cref{eqn: prop concentration WTS 2}.
\end{proof}

\section{Sampling from the \texorpdfstring{\(n\)}{n}-vector model}\label{sec: finite case}
This section discusses numerical sampling from the \(n\)-vector model \(p\) in \Cref{eqn: spherical model}. From \Cref{prop: distribution on homogeneous spaces}, it is shown that sampling \(X=[X_1, \ldots, X_k] \sim p\) can be done by sampling \(Q, R\) so that \((Q, R)\) are distributed according to the output of \(\Phi(X)\). Moreover, generating \(Q\) from the Haar measure of \(O(n)\) is efficient, e.g., by performing singular value decomposition on random matrices from the Gaussian orthogonal ensemble. Therefore, being able to sample from the distribution of \(R\) would allow one to sample from \(X \sim p\).

To see why it might be desirable to sample \(X \sim p\) from \(R\), we discuss two natural alternative directions.
One way is to directly sample \(X=[X_1, \ldots, X_k] \in \R^{n \times k}\) from the distribution function \(p\). However, directly sampling from \(p\) is quite cumbersome due to the spherical constraint that  \(\lVert X_i \rVert = 1\) for \(i = 1, \ldots, k\). Another proposal is to sample \(S = X^{\top}X\), as \Cref{theorem: main for S} provides a simple distribution function for \(S\). However, sampling from \(S\) is arguably more difficult than sampling from \(X\), as one would then need to perform sampling on the manifold of semidefinite matrices.

This section gives the theoretical background for two methods to sample \(R\). In the first proposed method, the goal is approximate sampling. Having solved the regularized SDP in \Cref{eqn: Goemans-Williamson with logdet}, one can take \(R = R^{\star}\) with \(R^{\star}\) in \Cref{thm: characterization of spherical model}. Henceforth, we generate new samples of \(Q\) from the Haar measure of \(O(n)\), and one can generate approximate samples of \(p\) by taking \(X = QR^{\star}\). The proposal is well-defined, and \Cref{sec: proof of spherical model characterization} justifies the proposal by proving \Cref{thm: characterization of spherical model} holds.

In the second proposed method, the goal is to perform MCMC sampling from \(R\). In \Cref{sec: distribution formula of R}, we derive the analytic formula for the distribution of \(R\). Our formula shows that the strictly upper triangular entries of \(R\) fully determine the distribution of \(R\). Moreover, we show that the strictly upper triangular part of \(R\) has a probability density function. While this work does not focus on the implementation of the exact sampling of \(R\), the formula for the density function would allow conventional MCMC samplers to be used, e.g., the Gibbs sampler \cite{geman1984stochastic, liu2001monte}.

\subsection{Proof of Theorem \ref{thm: characterization of spherical model}}\label{sec: proof of spherical model characterization}
Due to the results proven in \Cref{sec: background} and \Cref{sec: infinite case}, the proof of \Cref{thm: characterization of spherical model} is simple.
\begin{proof}
(Proof of \Cref{thm: characterization of spherical model})

\Cref{prop: distribution on homogeneous spaces} shows that (\ref{thm: characterization of spherical model pt 1})-(\ref{thm: characterization of spherical model pt 2}) in \Cref{thm: characterization of spherical model} is a simple consequence of the \(O(n)\) invariance of the \(n\)-vector model. By the property of the Gram-Schmidt decomposition, it follows that the term \(R\) for \((Q, R) = \Phi(X)\) coincides with the right Cholesky factor of \(S = X^{\top}X\). Therefore, it remains to show that the Cholesky factorization is stable, and then the remaining claims in \Cref{thm: characterization of spherical model} follow as a consequence of \Cref{thm: characterization of correlation function}.

Let \(X \sim p\) for \(p\) in \Cref{eqn: spherical model}. We let \(S = X^{\top}X\) and let \(R\) be the right Cholesky factor of \(S\). From \Cref{prop: uniqueness of f_beta maximizer} we have shown that \(S^{\star}_{\beta}\) is positive definite. We quote the following result directly from Theorem 1.1 in \cite{sun1991perturbation}: Let \(\kappa_{2}(S^{\star}_{\beta})\) be the condition number of \(S^{\star}_{\beta}\). We use \(\lVert \cdot \rVert_{2}\) to denote the matrix operator norm. Then, under the event that \(\lVert (S^{\star}_{\beta})^{-1}\rVert_{2}\lVert S - S^{\star}_{\beta} \rVert_{F} <\frac{1}{2}\), one has
\begin{equation}\label{eqn: stability of cholesky step 1}
    \frac{\lVert R - R^{\star} \rVert_{F} }{\lVert R^{\star} \rVert_{F}} \leq  \frac{
        \sqrt{2}\kappa_{2}(S^{\star}_{\beta})\lVert S - S^{\star}_{\beta} \rVert_{F}/\lVert S^{\star}_{\beta} \rVert_{F}
    }{
    1 + \sqrt{1 - 2\kappa_{2}(S^{\star}_{\beta})\lVert S - S^{\star}_{\beta} \rVert_{F}/\lVert S^{\star}_{\beta}\rVert_2}
    }.
\end{equation}

For the proof, we can discard the denominator on the right-hand side of \Cref{eqn: stability of cholesky step 1}. Thus, under the event that \(\lVert (S^{\star}_{\beta})^{-1}\rVert_{2}\lVert S - S^{\star}_{\beta} \rVert_{F} <\frac{1}{2}\), one has
\begin{equation*}
    \frac{\lVert R - R^{\star} \rVert_{F} }{\lVert R^{\star} \rVert_{F}} \leq  
        \sqrt{2}\kappa_{2}(S^{\star}_{\beta})\lVert S - S^{\star}_{\beta} \rVert_{F}/\lVert S^{\star}_{\beta} \rVert_{F}.
\end{equation*}

Therefore, for the constant \(C = \sqrt{2}\kappa_{2}(S^{\star}_{\beta})\lVert R^{\star} \rVert_{F}/\lVert S^{\star}_{\beta} \rVert_{F}\), which does not depend on \(n\), one has
\begin{equation}
    \label{eqn: stability of cholesky step 2}
    \lVert R - R^{\star} \rVert_{F} \leq C\lVert S - S^{\star}_{\beta} \rVert_{F}.
\end{equation}
From \Cref{thm: characterization of correlation function} we know that the probability of the event \(\lVert (S^{\star}_{\beta})^{-1}\rVert_{2}\lVert S - S^{\star}_{\beta} \rVert_{F} <\frac{1}{2}\) converges to one when \(n \to \infty\). Therefore, the inequality in \Cref{eqn: stability of cholesky step 2} holds with probability converging to one with \(n \to \infty\). Thus, we have proven
\begin{equation}
    \label{eqn: stability of cholesky step 3}
    \lim_{n \to \infty} \mathbb{P}\left[ \lVert R - R^{\star} \rVert_{F} < Cn^{-1/2 + b}\right] = 1.
\end{equation}
As \(Cn^{-1/2 + b/2} < n^{-1/2 + b}\) for \(n\) sufficiently large, we see that \Cref{eqn: stability of cholesky step 3} implies \Cref{eqn: main for R} in \Cref{thm: characterization of spherical model}.
\end{proof}

\subsection{Distribution formula of \texorpdfstring{\(R\)}{R}}\label{sec: distribution formula of R}

We use \(G_{U}^{+} \subset \R^{k \times k}\) to denote the space of \(k \times k\) upper-triangular matrices with positive diagonal entries. Let \(X = \left[X_1, \ldots, X_k\right] \sim p\) and let \(R\) be from \((Q, R) = \Phi(X)\). One can see from \Cref{theorem: main for S} that \(X\) is generically invertible, and so one can assume that \(\Phi(X)\) is always well-defined. It is clear that \(R\) is supported on \(G_{U}^{+}\). One sees that \(G_{U}^{+}\) is an open subset of \(\R^{k(k+1)/2}\), and one uses \(\mu_{R}\) to denote the Lebesgue measure on \(G_{U}^{+}\).
We prove the following statement on the distribution function of \(R\).

\begin{thm}\label{thm: main for R}
Let \(p\) be the \(n\)-vector model in \Cref{eqn: spherical model}. Define \(Q(X), R(X)\) so that \((Q(X), R(X)) := \Phi(X)\) is the unique output of Gram-Schmidt. For a continuous function \(f: G_{U}^{+} \to \R\), one has
\[
\int_{X \in \R^{n \times k}} dX\, p(X) f(R(X)) = \int_{R \in G_{U}^{+}} \mu_{R}(dR)p_R(R) f(R),
\]
where for \(R = \left[r_1, \ldots, r_k\right]\) one has
\begin{equation}\label{eqn: distribution of R}
p_R(R) = 2^{k}C_{n,k}\exp\left(\beta n\sum_{ij}\left<r_{i}, r_{j}\right>A_{ij}\right)
    \prod_{j= 1}^{k}R_{jj}^{n-j}\prod_{j= 1}^{k}\delta(1 - \lVert r_{j}\rVert^2 ).
\end{equation}
\end{thm}

We discuss how \Cref{thm: main for R} leads to an exact sampling strategy for \(R\). The formula in \Cref{eqn: distribution of R} implies that \(R_{jj}\) is completely determined by \((R_{ij})_{i < j}\) according to \[R_{jj} = \sqrt{1 - \sum_{i = 1}^{j-1} R_{ij}^2}.\] Therefore, one can marginalize out the redundant \((R_{jj})_{j = 1}^{k}\) variable. By a simple change of variable, one obtains the following (unnormalized) probability density function of \(\{R_{ij}\}_{i<j}\) with respect to the Lebesgue measure over \(\R^{k(k-1)/2}\):
\begin{equation}
\begin{aligned}
\label{eqn: sphere density}
&p_{U}(\{R_{ij}\}_{i<j}) = \exp\left(\beta n\sum_{j = 1}^{k}\sum_{i,i'=1}^{j}R_{ij} R_{i'j}A_{ii'}\right)
    \prod_{j= 1}^{k}R_{jj}^{n-j-1} \text{, where}\\
    &R_{jj} = \sqrt{1 - \sum_{i = 1}^{j-1} R_{ij}^2}, \text{ for  \(j = 1,\ldots, k\)}.
\end{aligned}
\end{equation}
One can thus perform MCMC sampling on the joint variable \(\{R_{ij}\}_{i<j}\) by \Cref{eqn: sphere density}. A slight difficulty is that the variables \(\{R_{ij}\}_{i<j}\) needs to satisfy \(\sum_{i = 1}^{j-1} R_{ij}^2 \leq 1\) for any \(j = 1,\ldots, k\). One possible proposal is to perform the Gibbs sampler \cite{geman1984stochastic} by only updating one \(R_{ij}\) variable at a time. By fixing all other entries and only updating \(R_{ij}\), one sees that the support for \(R_{ij}\) is an interval and can be easily calculated. As the conditional distribution of \(R_{ij}\) is known through \(p_{U}\) in \Cref{eqn: sphere density}, updating \(R_{ij}\) is simple.

We now prove \Cref{thm: main for R}. Similar to \Cref{theorem: main for S}, we use an intermediate result for \(O(n)\)-invariant distributions. The result is proven in \cite{eaton1983multivariate} and we quote it here.
\begin{prop}\label{prop: density on homogeneous spaces version R}
(Proposition 7.5 of \cite{eaton1983multivariate})
Let \(X\in \R^{n \times k}\) be a random matrix with distribution \(X \sim \nu \). Suppose that \(\nu\) satisfies the assumption in \Cref{prop: density on homogeneous spaces}. In other words, we assume \(\nu\) is \(O(n)\)-invariant and the density for \(X\) is defined by \(g(X) = h(X^{\top}X)\). Then \(R = R(X)\) has the following density \(g_R\) with respect to \(\mu_{R}\):
\[
g_R(R) = 2^{k}C_{n,k}h(R^{\top}R)\prod_{j = 1}^{k}R_{jj}^{n-j}.
\]
In particular, for a continuous function \(f: G_{U}^{+} \to \R\), one has
\[
\int_{X \in \R^{n \times k}} \nu(dX)\, f(R(X)) = \int_{R \in G_{U}^{+}} \mu_{R}(dR)g_R(R) f(R).
\]
\end{prop}

The proof of \Cref{thm: main for R} is similar to \Cref{theorem: main for S} and is done by utilizing \Cref{prop: sphere dirac delta}. We give the proof here.

\begin{proof}
    (Proof of \Cref{thm: main for R})

We use the construction of \(\nu\) as in the proof for \Cref{theorem: main for S}. As \(f\) is continuous, we use \Cref{prop: sphere dirac delta} to write
\begin{align*}
&\int_{X \in \R^{n \times k}} dX \, p(X) f(R(X)) \\ = &\lim_{t_1, \ldots, t_{k} \to 0}\int_{X \in \R^{n \times k}} dX \exp(\beta n \mathrm{tr}(AX^{\top}X))f(R(X))\prod_{i=1}^{k}
\frac{1}{t_i}\chi{\left(\lVert x_i \rVert \in [1, 1 + t_i]\right)}.
\end{align*}

By the construction of \(\nu\), it follows that \Cref{prop: density on homogeneous spaces version R} applies and we can write
\begin{align*}
&\int_{X \in \R^{n \times k}}dX\, p(X) f(R(X))\\
=&\lim_{t_1, \ldots, t_{k} \to 0}\int_{X \in \R^{n \times k}} \nu(dX)f(R(X))\prod_{i=1}^{k}
\frac{1}{t_i}\chi{\left(\lVert x_i \rVert \in [1, 1 + t_i]\right)}.\\
=&\lim_{t_1, \ldots, t_{k} \to 0}\int_{R \in G_{U}^{+}}2^{k}C_{n, k}\exp(\beta n \mathrm{tr}(AR^{\top}R))f(R)\prod_{j=1}^{k}R_{jj}^{n-j} 
    \prod_{i=1}^{k}
\frac{1}{t_i}\chi{\left(\lVert r_i \rVert \in [1, 1 + t_i]\right)}.
\end{align*}

We then exchange the order of limit and integration. Applying \Cref{prop: sphere dirac delta} again, one has
\begin{align*}
&\int_{X \in \R^{n \times k}}dX\, p(X) f(R(X))\\
=&\int_{R \in G_{U}^{+}}2^{k}C_{n, k}\exp(\beta n \mathrm{tr}(AR^{\top}R))f(R) \prod_{j=1}^{k}R_{jj}^{n-j}
\left(\lim_{t_1, \ldots, t_{k} \to 0}\prod_{i=1}^{k}\frac{1}{t_i}\chi{\left(\lVert r_i \rVert \in [1, 1 + t_i]\right)}\right)\\
=&\int_{R \in G_{U}^{+}}2^{k}C_{n, k}\exp(\beta n \mathrm{tr}(AR^{\top}R))f(R)\prod_{j=1}^{k}R_{jj}^{n-j} \prod_{i=1}^{k}\delta(1 - \lVert r_i \rVert ),
\end{align*}
where the last equality holds due to Fubini's theorem. Thus, we are done.
\end{proof}

\section{Discussion}\label{sec: conclusion}

We study a spherical \(n\)-vector model under a pairwise interaction potential. The model is a lifted spin model, and we study it under a finite temperature regime. We show that the model is exactly solvable in the \(n \to \infty\) limit. A sampling strategy for the \(n\)-vector model is discussed. A future research direction is to study the convergence of the normalized free energy \(Q_{n}(\beta)\) with a tighter non-asymptotic bound. An open question is whether the samples of the \(n\)-vector model for \(n > k\) are related to the Ising model at \(n = 1\). While a relationship at \(\beta \to \infty\) is exhibited by the work by Goemans and Williamson in \cite{goemans1995improved} through the SDP approximation ratio, it remains to be seen whether a similar result holds for a finite temperature setting. 

An interesting potential application of the \(n\)-vector model is to use the hyperspace projection of the Goemans-Williamson scheme to obtain approximate samples of the Ising model. Explicitly, for each sample \([x_1, \ldots, x_{k}] \sim p\) from the \(n\)-vector model \(p\), one can generate a random direction \(v \in \R^{n}\) and take \(s_i = \mathrm{sign}(\left<v, x_{i}\right>)\) to generate a sample \([s_1, \ldots, s_k]\) on the hypercube \(\{-1, 1\}^{k}\). The samples obtained in this fashion might serve as a good initialization when one performs MCMC on the corresponding Ising model with the same pairwise interaction matrix. Therefore, an open research question is whether such a proposal leads to faster mixing in practice.

\bibliographystyle{siamplain}
\bibliography{reference}

\begin{thebibliography}{10}

\bibitem{abramowitz1988handbook}
{\sc M.~Abramowitz, I.~A. Stegun, and R.~H. Romer}, {\em Handbook of mathematical functions with formulas, graphs, and mathematical tables}, 1988.

\bibitem{allen2019convergence}
{\sc Z.~Allen-Zhu, Y.~Li, and Z.~Song}, {\em A convergence theory for deep learning via over-parameterization}, in International conference on machine learning, PMLR, 2019, pp.~242--252.

\bibitem{arora2019fine}
{\sc S.~Arora, S.~Du, W.~Hu, Z.~Li, and R.~Wang}, {\em Fine-grained analysis of optimization and generalization for overparameterized two-layer neural networks}, in International conference on machine learning, PMLR, 2019, pp.~322--332.

\bibitem{boumal2016nonconvex}
{\sc N.~Boumal}, {\em Nonconvex phase synchronization}, SIAM Journal on Optimization, 26 (2016), pp.~2355--2377.

\bibitem{candes2010power}
{\sc E.~J. Cand{\`e}s and T.~Tao}, {\em The power of convex relaxation: Near-optimal matrix completion}, IEEE transactions on information theory, 56 (2010), pp.~2053--2080.

\bibitem{du2018power}
{\sc S.~Du and J.~Lee}, {\em On the power of over-parametrization in neural networks with quadratic activation}, in International conference on machine learning, PMLR, 2018, pp.~1329--1338.

\bibitem{du2018gradient}
{\sc S.~S. Du, X.~Zhai, B.~Poczos, and A.~Singh}, {\em Gradient descent provably optimizes over-parameterized neural networks}, arXiv preprint arXiv:1810.02054,  (2018).

\bibitem{eaton1983multivariate}
{\sc M.~L. Eaton}, {\em Multivariate statistics: a vector space approach.}, JOHN WILEY \& SONS, INC., 605 THIRD AVE., NEW YORK, NY 10158, USA, 1983, 512,  (1983).

\bibitem{evans2018measure}
{\sc L.~Evans}, {\em Measure theory and fine properties of functions}, Routledge, 2018.

\bibitem{ge2016matrix}
{\sc R.~Ge, J.~D. Lee, and T.~Ma}, {\em Matrix completion has no spurious local minimum}, Advances in neural information processing systems, 29 (2016).

\bibitem{geman1984stochastic}
{\sc S.~Geman and D.~Geman}, {\em Stochastic relaxation, gibbs distributions, and the bayesian restoration of images}, IEEE Transactions on pattern analysis and machine intelligence,  (1984), pp.~721--741.

\bibitem{goemans1995improved}
{\sc M.~X. Goemans and D.~P. Williamson}, {\em Improved approximation algorithms for maximum cut and satisfiability problems using semidefinite programming}, Journal of the ACM (JACM), 42 (1995), pp.~1115--1145.

\bibitem{li2018learning}
{\sc Y.~Li and Y.~Liang}, {\em Learning overparameterized neural networks via stochastic gradient descent on structured data}, Advances in neural information processing systems, 31 (2018).

\bibitem{liu2022loss}
{\sc C.~Liu, L.~Zhu, and M.~Belkin}, {\em Loss landscapes and optimization in over-parameterized non-linear systems and neural networks}, Applied and Computational Harmonic Analysis, 59 (2022), pp.~85--116.

\bibitem{liu2001monte}
{\sc J.~S. Liu and J.~S. Liu}, {\em Monte Carlo strategies in scientific computing}, vol.~10, Springer, 2001.

\bibitem{makeenko2002methods}
{\sc Y.~Makeenko}, {\em Methods of contemporary gauge theory}, Cambridge University Press, 2002.

\bibitem{pastur1982disordered}
{\sc L.~Pastur}, {\em Disordered spherical model}, Journal of Statistical Physics, 27 (1982), pp.~119--151.

\bibitem{stanley1969exact}
{\sc H.~Stanley}, {\em Exact solution for a linear chain of isotropically interacting classical spins of arbitrary dimensionality}, Physical Review, 179 (1969), p.~570.

\bibitem{stanley1968spherical}
{\sc H.~E. Stanley}, {\em Spherical model as the limit of infinite spin dimensionality}, Physical Review, 176 (1968), p.~718.

\bibitem{sun1991perturbation}
{\sc J.-G. Sun}, {\em Perturbation bounds for the cholesky and qr factorizations}, BIT Numerical Mathematics, 31 (1991), pp.~341--352.

\bibitem{vandenberghe1996semidefinite}
{\sc L.~Vandenberghe and S.~Boyd}, {\em Semidefinite programming}, SIAM review, 38 (1996), pp.~49--95.

\bibitem{zou2019improved}
{\sc D.~Zou and Q.~Gu}, {\em An improved analysis of training over-parameterized deep neural networks}, Advances in neural information processing systems, 32 (2019).

\end{thebibliography}
\end{document}